\documentclass[11pt]{amsart}
\usepackage{amscd}
\usepackage{amsmath}
\usepackage{amsxtra}
\usepackage{amsfonts}
\usepackage{amssymb}
\usepackage{bbm}

\newtheorem{theorem}{Theorem}[section]
\newtheorem{corollary}[theorem]{Corollary}
\newtheorem{lemma}[theorem]{Lemma}
\newtheorem{proposition}[theorem]{Proposition}

\newtheorem{conjecture}[theorem]{Conjecture}
\theoremstyle{definition}
\newtheorem{definition}[theorem]{Definition}
\newtheorem{remark}[theorem]{Remark}

\newtheorem{example}[theorem]{Example}
\theoremstyle{remark}

\renewcommand{\theclaim}{\textup{\theclaim}}

\newtheorem*{acknowledgements}{Acknowledgements}

\numberwithin{equation}{section}

\def\openone

{\mathchoice

{\hbox{\upshape \small1\kern-3.3pt\normalsize1}}

{\hbox{\upshape \small1\kern-3.3pt\normalsize1}}

{\hbox{\upshape \tiny1\kern-2.3pt\SMALL1}}

{\hbox{\upshape \Tiny1\kern-2pt\tiny1}}}

\makeatletter

\newbox\ipbox

\newcommand{\ip}[2]{\left\langle #1\, , \,#2\right\rangle}
\newcommand{\diracb}[1]{\left\langle #1\mathrel{\mathchoice

{\setbox\ipbox=\hbox{$\displaystyle \left\langle\mathstrut
#1\right.$}

\vrule height\ht\ipbox width0.25pt depth\dp\ipbox}

{\setbox\ipbox=\hbox{$\textstyle \left\langle\mathstrut
#1\right.$}

\vrule height\ht\ipbox width0.25pt depth\dp\ipbox}

{\setbox\ipbox=\hbox{$\scriptstyle \left\langle\mathstrut
#1\right.$}

\vrule height\ht\ipbox width0.25pt depth\dp\ipbox}

{\setbox\ipbox=\hbox{$\scriptscriptstyle \left\langle\mathstrut
#1\right.$}

\vrule height\ht\ipbox width0.25pt depth\dp\ipbox}

}\right. }

\newcommand{\dirack}[1]{\left. \mathrel{\mathchoice

{\setbox\ipbox=\hbox{$\displaystyle \left.\mathstrut
#1\right\rangle$}

\vrule height\ht\ipbox width0.25pt depth\dp\ipbox}

{\setbox\ipbox=\hbox{$\textstyle \left.\mathstrut
#1\right\rangle$}

\vrule height\ht\ipbox width0.25pt depth\dp\ipbox}

{\setbox\ipbox=\hbox{$\scriptstyle \left.\mathstrut
#1\right\rangle$}

\vrule height\ht\ipbox width0.25pt depth\dp\ipbox}

{\setbox\ipbox=\hbox{$\scriptscriptstyle \left.\mathstrut
#1\right\rangle$}

\vrule height\ht\ipbox width0.25pt depth\dp\ipbox}

} #1\right\rangle}

\newcommand{\cj}[1]{\overline{#1}}

\newcommand{\bz}{\mathbb{Z}}

\newcommand{\br}{\mathbb{R}}
\newcommand{\bc}{\mathbb{C}}

\newcommand{\bn}{\mathbb{N}}

\newcommand{\beq}{\begin{equation}}
\newcommand{\eeq}{\end{equation}}

\def\blfootnote{\xdef\@thefnmark{}\@footnotetext}

\renewcommand{\mod}{\operatorname{mod}}

\input xy
\xyoption{all}
\usepackage{amssymb}

\def\-{^{-1}}

\def\ty{\emptyset}

\begin{document}

\title[Weighted Fourier frames on fractal measures]{Weighted Fourier frames on fractal measures}
\author{Dorin Ervin Dutkay}

\address{[Dorin Ervin Dutkay] University of Central Florida\\
	Department of Mathematics\\
	4000 Central Florida Blvd.\\
	P.O. Box 161364\\
	Orlando, FL 32816-1364\\
U.S.A.\\} \email{Dorin.Dutkay@ucf.edu}

\author{Rajitha Ranasinghe}

\address{[Rajitha Ranasinghe] University of Central Florida\\
	Department of Mathematics\\
	4000 Central Florida Blvd.\\
	P.O. Box 161364\\
	Orlando, FL 32816-1364\\
U.S.A.\\} \email{rajitha13@knights.ucf.edu }

\thanks{} 

\begin{abstract}
We generalize an idea of Picioroaga and Weber \cite{PW15} to construct Paseval frames of weighted exponential functions for self-affine measures. 
 
\end{abstract}
\maketitle \tableofcontents

\section{Introduction}
In \cite{JP98}, Jorgensen and Pedersen proved that there exist singular measures $\mu$ which are spectral, that is, there exists a sequence of exponential functions which form an orthonormal basis for $L^2(\mu)$. Their example is based on the Cantor set with scale $4$ and digits $0$ and $2$:
$$C_4=\left\{\sum_{k=1}^\infty \frac{a_k}{4^k} : a_k\in\{0,2\}\right\}.$$
The spectral measure $\mu_4$ is the restriction to the set $C_4$ of the Hausdorff measure with dimension $\frac12$. It can also be seen as the invariant measure for the iterated function system
$$\tau_0(x)=\frac{x}4,\quad \tau_2(x)=\frac{x+2}4.$$
Jorgensen and Pedersen proved that the set of exponential functions 
$$\left\{e_\lambda : \lambda=\sum_{k=0}^n 4^k l_k, l_k\in\{0,1\}, n\in \bn\right\}$$
is an orthonormal basis for $\mu_4$ (here $e_\lambda(x)=e^{2\pi i \lambda x}$). 

Many other examples of spectral singular measures were constructed since (see e.g., \cite{Str00, LaWa02, DJ06,DJ07d}), most of them are based on affine iterated function systems.

\begin{definition}\label{defifs}

For a given expansive $d\times d$ integer matrix $R$ and a finite set of integer vectors $B$ with cardinality $|B| =: N$, we define the {\it affine iterated function system} (IFS) $\tau_b(x) = R^{-1}(x+b)$, $x\in \br^d, b\in B$. The {\it self-affine measure} (with equal weights) is the unique probability measure $\mu = \mu(R,B)$ satisfying
\begin{equation}\label{self-affine}
\mu(E) = \frac1N\sum_{b\in B}  \mu (\tau_b^{-1} (E)),\mbox{ for all Borel subsets $E$ of $\br^d$.}
\end{equation}
This measure is supported on the {\it attractor} $X_B$ which is the unique compact set that satisfies
$$
X_B= \bigcup_{b\in B} \tau_b(X_B).
$$
The set $X_B$ is also called the {\it self-affine set} associated with the IFS, and it can be described as 
$$X_B=\left\{\sum_{k=1}^\infty R^{-k}b_k : b_k\in B\right\}.$$

 One can refer to \cite{Hut81} for a detailed exposition of the theory of iterated function systems. We say that $\mu = \mu(R,B)$ satisfies the {\it no overlap condition} if
$$
\mu(\tau_{b}(X_B)\cap \tau_{b'}(X_B))=0, \ \forall b\neq b'\in B.
$$

For $\lambda\in\br^d$, define
$$e_\lambda(x)=e^{2\pi i\lambda\cdot x},\quad(x\in \br^d).$$

For a Borel probability measure $\mu$ on $\br^d$ we define its {\it Fourier transform} by
$$\widehat\mu(t)=\int e^{2\pi it\cdot x}\,d\mu,\quad(t\in\br^d).$$

A {\it frame} for a Hilbert space $H$ is a family $\{e_i\}_{i\in I}\subset H$ such that there exist constants $A,B>0$ such that for all $v\in H$,
$$A\|v\|^2\leq \sum_{i\in I}|\ip{v}{e_i}|^2\leq B\|v\|^2.$$
The largest $A$ and smallest $B$ which satisfy these inequalities are called the {\it frame bounds}. The frame is called a {\it Parseval} frame if both frame bounds are $1$. 

 \end{definition}

In \cite{PW15}, Picioroaga and Weber introduced an interesting idea for the construction of weighted exponential frames for the Cantor set $C_4$ in Jorgensen and Pedersen's example. Weighted means that the exponential function is multiplied by a constant. The general idea is the following: first construct a dilation of the space $L^2(\mu_4)$ by choosing another measure $\mu'$ (in their case, the Lebesgue measure on [0,1]) and considering the Hilbert space $L^2(\mu_4\times\mu')$. The space $L^2(\mu_4)$ can be regarded as a subspace of this Hilbert space, as the space of functions that depend only on the first variable. Then they construct an orthonormal set of functions in $L^2(\mu_4\times \mu')$ which project onto a Parseval frame of weighted exponential functions in $L^2(\mu_4)$. The general result is a basic fact in frame theory (see e.g., \cite{MR1242070}):

\begin{lemma}\label{lemframe}
Let $H$ be a Hilbert space, $V\subset K$, closed subspaces and let $P_V$ be the orthogonal projection onto $V$. If $\{e_i\}_{i\in I}$ is an orthonormal basis for $K$, then $\{P_Ve_i\}_{i\in I}$ is a Parseval frame for $V$.
\end{lemma}

To construct the orthonormal set in $L^2(\mu_4\times \mu')$ a representation of the Cuntz algebra $\mathcal O_M$ is constructed. Recall that the Cuntz algebra is generated by $M$ isometries with the properties 
\begin{equation}
S_i^*S_j=\delta_{ij}I,\quad \sum_{i=1}^MS_iS_i^*=I.
\label{eqcuntz}
\end{equation}

Then the Cuntz isometries are applied to the constant function $\mathbf 1$:
$$\left\{S_{\omega_1}\dots S_{\omega_k}\mathbf 1 : \omega_i\in\{1,\dots, M\}, n\in \bn\right\},$$
to obtain the orthonormal set which is then projected onto the subspace $L^2(\mu_4)$ and produces a Parseval frame of weighted exponential functions. This is based on a modification of Theorem 3.1 from \cite{MR3103223}:

\begin{theorem}\cite{MR3103223}\label{thoc}
Let $H$ be a Hilbert space, $K\subset H$ a closed subspace, and $(S_i)_{i=1}^{M}$ be a representation of the Cuntz algebra $\mathcal O_M$. Let $\mathcal E$ be an orthonormal set in $H$ and $f:X\rightarrow K$ be a norm continuous function on a topological space $X$ with the following properties:
\begin{enumerate}
	\item $\mathcal E=\cup_{i=1}^MS_{i}\mathcal E$, where the union is disjoint. 
	\item $\cj{\mbox{span}}\{f(t): t\in X\}=K$ and $\|f(t)\|=1$, for all $t\in X$. 
	\item There exist functions $\mathfrak m_i:X\rightarrow \bc$, $g_i:X\rightarrow X$, $i=1,\dots,M$ such that 
	\begin{equation}
	S_i^*f(t)=\mathfrak m_i(t)f(g_i(t)),\quad (t\in X).
	\label{eqtoc1}
	\end{equation}
	\item There exists $c_0\in X$ such that $f(c_0)\in \cj{\mbox{span}}\mathcal E$.
	\item The only function $h\in C(X)$ with $0\leq h\leq 1$, $h(c)=1$, for all $c\in \{x\in X : f(x)\in \cj{\mbox{span}}\mathcal E\}$, and 
	\begin{equation}
	h(t)=\sum_{i=1}^M |\mathfrak m_i(t)|^2h(g_i(t)),\quad (t\in X),
	\label{eqtoc2}
	\end{equation}
	is the constant function $\mathbf 1$.
\end{enumerate}
Then $K\subset \cj{\mbox{span}}\mathcal E$. 
\end{theorem}

In this paper, we generalize and refine the construction from \cite{PW15} to build Parseval frames of weighted exponential functions (in other words weighted Fourier frames) for the Hilbert space $L^2(\mu(R,B))$ associated to the invariant measure of an affine iterated function system as in Definition \ref{defifs}.

Our main result is:

\begin{theorem}\label{th1}
Let $R$ be a $d\times d$ expansive integer matrix, $B$ be a finite subset of $\bz^d$, $0\in B$, $N:=|B|$, and suppose that the measure $\mu=\mu(R,B)$ has no overlap. Assume that there exists a finite set $L\subset \bz^d$ with $0\in L$, $|L|=:M$, and complex numbers $(\alpha_{l})_{l\in L}$,  such that the following properties are satisfied:
\begin{enumerate}
	\item $\alpha_{0}=1$.
	\item The matrix 
	\begin{equation}
	T:=\frac{1}{\sqrt{N}} \left(  e^{2 \pi i {({R}^{T}})^{-1} l \cdot b} \alpha_{l} \right)_{l\in L, b\in B}
	\label{eqt1.1}	
	\end{equation}
	is an isometry, i.e., $T^{*}T=I_{N}$, i.e., its columns are orthonormal. 	
	\item The only entire function $h$ on $\bc^d$ with the property that $0\leq h\leq 1$ on $\br^d$, $h(0)=1$ and 
	\begin{equation}
	\sum_{l\in L}|\alpha_{l}|^2\left|m_B((R^T)^{-1}(t-l))\right|^2h((R^T)^{-1}(t-l))=h(t),\quad(t\in \br^d)
	\label{eqt1.3}
	\end{equation}
	is the constant function $\mathbf 1$. Here,
	\begin{equation}
	m_B(x)=\frac{1}{N}\sum_{b\in B}e^{2\pi i b\cdot x},\quad(x\in \br^d).
	\label{eqt1.4}
	\end{equation}

\end{enumerate}

Let $\Omega(L)$ be the set of finite words with digits in $L$: 
$$\Omega(L)=\{\ty\}\cup\{l_0\dots l_k : l_i\in L, l_k\neq 0\},\quad(\ty\mbox{ represents the empty word}).
$$Then the set 
\begin{equation}
\left\{\left( \prod_{i=0}^k\alpha_{l_i}\right) e_{l_0+R^T l_1+\dots +(R^T)^{k}l_k} : l_0\dots l_k\in\Omega(L)\right\}
\label{eqt1.5}
\end{equation}
is a Parseval frame for $L^2(\mu(R,B))$.
\end{theorem}

\begin{remark}
The non-overlap condition is satisfied if the elements in $B$ are incongruent modulo $R\bz^d$, see \cite[Theorem 1.7]{DuLa15}.
\end{remark}

In section 2, we present the proof of our main result. In section 3 we include some remarks on the hypotheses of Theorem \ref{th1} and present some restrictions that they entail. In section 4, we focus on the case of dimension $d=1$ and rephrase the condition (iii) in Theorem \ref{th1} in simpler terms. We conclude the paper with some examples and a conjecture, in section 5. 
\section{Proof}

\begin{proof}[Proof of Theorem \ref{th1}]
We will transform our hypotheses into the setting in \cite{PW15} and after that we follow the ideas from \cite{PW15} in this, more general, context.

 Consider now another system with $R^{\prime}$ a $d^{\prime} \times d^{\prime}$ expansive integer matrix, $B^{\prime} \subset \mathbb{Z}^{d^{\prime}}, 0 \in B^{\prime}, |B^{\prime}|=N^{\prime}$ and let $\mu^{\prime}$ be the invariant measure for the iterated function system 
$$ \tau_b^{\prime}(x^{\prime})={R^{\prime}}^{-1}(x^{\prime}+b^{\prime}) \ \ \ ( x^{\prime} \in \mathbb{R}^d, b^{\prime} \in B^{\prime})$$
and let $X_{B^{\prime}}$ be its attractor. Assume also that the system has no overlap, i.e.,
$$\mu^{\prime}\left( \tau'_{b^{\prime}_i}(X_{B^{\prime}}) \cap \tau'_{b^{\prime}_j}(X_{B^{\prime}})  \right)=0, \ \ \textnormal{for all} \ \ b^{\prime}_i \neq b^{\prime}_j \ \ \textnormal{in} \ \ B^{\prime}.$$
As we will see, it is not important how we pick the matrix $R'$ and the digits $B'$, the only thing that matters is the number of digits in $B'$; for example we can just take $B'$ to be $\{0,1,\dots,N'-1\}$ and $R'=N'$. We keep a higher level of generality to see how far this generalization goes. We only require that $NN'\geq M$. We can identify $L$ with a subset $L'$ of $B\times B'$, by some injective function $\iota$, in such a way that $0$ from $L$ corresponds to $(0,0)$ from $B\times B'$, and we define $l(b,b')=  l$ if $(b,b')=\iota(l)$, $l(b,b')=0$ if $(b,b')\not\in L'$, and 
$\alpha_{(b,b')}=\alpha_{l}$ if $(b,b')=\iota(l)$ and $\alpha_{(b,b')}=0$ if $(b,b')\not\in L'$. In other words, we complete the matrix $T$ in (ii) with some zero rows, so that the rows are now indexed by $B\times B'$, and of course the isometry property is preserved, and $\alpha_{(0,0)}=0$, $l(0,0)=0$. Thus, the properties (i)--(iii) are satisfied with the indexing set $L$ replaced by $B\times B'$, the numbers $l$ from $L$ replaced by the numbers $l(b,b')$, and the numbers $\alpha_l$ replaced by the numbers $\alpha_{(b,b')}$. 

Next, we construct the numbers $a_{(b,b'),(c,c')}$, $(b,b'),(c,c')\in B\times B'$ with the following properties:
\begin{enumerate}
	\item The matrix 
	\begin{equation}
	\frac{1}{\sqrt{NN'}}\left(a_{(b,b'),(c,c')}e^{2\pi i (R^T)^{-1}l(b,b')\cdot c}\right)_{(b,b'),(c,c')\in B\times B'}
	\label{eqp1.1}
	\end{equation}
	is unitary and the first row is constant $\frac{1}{\sqrt{NN'}}$, so $a_{(0,0),(c,c')}=1$ for all $(c,c')\in B\times B'$.
	\item For all $(b,b')\in B\times B'$, $c\in B$
	\begin{equation}
	\frac{1}{N'}\sum_{c'\in B'}a_{(b,b'),(c,c')}=\alpha_{(b,b')}.
	\label{eqp1.2}
	\end{equation}
\end{enumerate}

\begin{remark}
In \cite{PW15}, the authors begin their construction with the numbers $a_{(b,b'),(c,c')}$. But, as we see here, this is not necessary, we can start with the numbers $\alpha_l$, which are directly connected to the measure $\mu$ and do not involve the auxiliary measure $\mu'$. 
\end{remark}

Let $t_{(b,b^{\prime}),c}=\frac{1}{\sqrt{N}} e^{2 \pi i ({R}^T)^{-1} l(b,b^{\prime}) \cdot c } \alpha_{(b,b^{\prime})}.$ Note that the vectors $\{t_{\cdot, c}\}_{c\in B}$, in $\mathbb{C}^{NN^{\prime}}$ are orthonormal. Therefore, we can define some vectors $t_{\cdot , d}, \ d \in \{ 1, . . . , NN^{\prime}-N \}$ such that   
$$\{ t_{\cdot , c} : c \in B \} \cup \{ t_{\cdot , d} : d \in \{ 1, . . . , NN^{\prime}-N \} \}$$
is an orthonormal basis for $\mathbb{C}^{NN^{\prime}}.$

For $c \in B$ , define the vectors in $\mathbb{C}^{NN^{\prime}}$ by 
$$e_{c}(c_1, {c^{\prime}_1})=\frac{1}{\sqrt{N^{\prime}}} \delta_{cc_1} \ \ \  \left( (c_1, {c^{\prime}_1}) \in B \times B^{\prime}  \right).$$  

It is easy to see that these vectors are orthonormal in $\mathbb{C}^{NN^{\prime}}$, therefore we can complete them to an orthonormal basis for $\bc^{NN'}$ with some vectors $e_d \  , \ d \in \{ 1, . . . , NN^{\prime}-N \}$. Define now 

$$ {s}_{(b,b^{\prime})}=\sum\limits_{c \in B} t_{(b,b^{\prime}), c} e_{c}+\sum\limits_{d=1}^{NN^{\prime}-1} t_{(b,b^{\prime}), d} e_{d}.$$

Since the matrix with columns $t_{\cdot, c}$  and $t_{\cdot, d}$ has orthonormal columns, it is unitary. So it has orthogonal rows. So the vectors $t_{(b,b^{\prime}), \cdot}$ are orthonormal, therefore the vectors ${s}_{(b,b^{\prime})}$ are orthonormal. Also, since  $\alpha_{(0,0)}=1$  and $l_{(0,0)}=0$  we have that $t_{(0,0),c}=\frac{1}{\sqrt{N}}$  for all $c \in B.$ But then 
$$\sum\limits_{c \in B} |t_{(0,0),c}|^2 =1=\| t_{(0,0),\cdot} \|^2.$$
So $t_{(0,0),d}=0$   for $d \in \{ 1, . . . , NN^{\prime}-N \}.$ Therefore, for all $(c_1, {c^{\prime}_1}) \in B \times B^{\prime}:$
$${s}_{(0,0)}(c_1, {c^{\prime}_1})=\sum\limits_{c \in B} \frac{1}{\sqrt{N}} \ \frac{1}{\sqrt{N^{\prime}}} \delta_{cc_1}=\frac{1}{\sqrt{N N^{\prime}}}.$$
The vectors $\{ e_c : c \in B  \}$ span the subspace
$$\mathcal{M}=\{ \left( X{(c,c^{\prime})} \right)_{(c,c^{\prime}) \in B \times B^{\prime} } : \textnormal{$X$ does not depend on} \ c^{\prime}   \}.$$
So the vectors $\{ e_d : d \in \{ 1, . . . , NN^{\prime}-N \} \}$ are ortogonal to $\mathcal{M}.$ Let $P_{\mathcal{M}}$ be the projection onto $\mathcal{M}.$

Note that, for $X \in \mathbb{C}^{N N^{\prime}},$ we have
$$\left( P_{\mathcal{M}}X  \right){(c_1,c_1^{\prime})} = \sum_{c\in B}\ip{X}{e_c}e_c(c_1,c_1')=\frac{1}{N^{\prime}} \sum\limits_{c_1^{\prime} \in B^{\prime}} X{(c_1, c_1^{\prime})}.$$
Also
$$
\left( P_{\mathcal{M}} s_{(b,b^{\prime})} \right)(c_1, c_1^{\prime}) = \sum\limits_{c \in B} t_{(b,b^{\prime}), c} e_c (c_1, c_1^{\prime}) 
=\sum_{c \in B} t_{(b,b^{\prime}), c}\frac1{\sqrt{N'}}\delta_{cc_1} =\frac{1}{\sqrt{N^{\prime}}} t_{(b,b^{\prime}), c_1}.$$
Define now
$$a_{(b,b^{\prime}),(c,c^{\prime})} :=\sqrt{NN^{\prime}} e^{-2 \pi i \left( R^{T} \right)^{-1} l(b,b^{\prime}) \cdot c}  {s}_{(b,b^{\prime})} (c,c^{\prime}).$$
Then we have,
$$a_{(0,0), (c,c^{\prime})}=1 \ \textnormal{for all} \ (c,c^{\prime}).$$
The matrix
$$\frac{1}{\sqrt{N N^{\prime}}} \left( a_{(b,b^{\prime}), (c,c^{\prime})} e^{2 \pi i \left( R^{T} \right)^{-1} l(b,b^{\prime}) \cdot c} \right)_{(b,b^{\prime}), (c,c^{\prime})}$$
is the matrix with rows ${s}_{(b,b^{\prime})}.$ So it is unitary.

$$\frac{1}{N^{\prime}} \sum\limits_{c^{\prime} \in B^{\prime}} a_{(b,b^{\prime}), (c,c^{\prime})} =\frac1{N'}\sum_{c'\in B'}\sqrt{NN'}e^{-2\pi i(R^T)^{-1}l(b,b')\cdot c}s_{(b,b')}(c,c')$$$$
=\sqrt{NN^{\prime}} e^{-2 \pi i \left( R^{T} \right)^{-1} l(b,b^{\prime}) \cdot c} \left( P_{\mathcal{M}} {s}_{(b,b^{\prime})} \right){(c,c^{\prime})} 
=\sqrt{NN^{\prime}}e^{-2 \pi i \left( R^{T} \right)^{-1} l(b,b^{\prime}) \cdot c} \cdot\frac{1}{\sqrt{N'}} t_{(b,b^{\prime}), c} 
=\alpha_{(b,b^{\prime})}.$$
Thus, the conditions (i) and (ii) for the numbers $a_{(b,b'),(c,c')}$ are satisfied. 

Next, we construct some Cuntz isometries $S_{(b,b')}$, $(b,b')\in B\times B'$ in the dilation space $L^2(\mu\times\mu')$ and with them we construct an orthonormal set, by applying the Cuntz isometries to the function $\mathbf 1$.

Define now the maps $\mathcal{R}:X_B \rightarrow X_B$ by 
\begin{equation}
\mathcal{R}x={R}x-b \ \ \textnormal{if} \ \ x \in \tau_{b}(X_{B})
\end{equation}
and ${\mathcal{R}}^{\prime}:X_B^{\prime} \rightarrow X_B^{\prime}$ by 
\begin{equation}
\mathcal{R}^{\prime}x^{\prime}={R}^{\prime}x^{\prime}-b^{\prime} \ \ \textnormal{if} \ \ x^{\prime} \in \tau_{b^{\prime}}(X_{B^{\prime}}).
\end{equation} 
Note that $\mathcal{R}(\tau_b x)=x$ for all $x \in X_B$ and $\mathcal{R}^{\prime}(\tau'_b x^{\prime})=x^{\prime}$ for all $x^{\prime} \in X_{B^{\prime}}$. The non-overlap condition guarantees that the maps are well defined. 

 Next we consider the cartesian product of the two iterated function systems and define the maps
\begin{equation}
\Upsilon_{(b,b^{\prime})}(x,x^{\prime})=\left( R^{-1}(x+b), R'^{-1}(x^{\prime}+b^{\prime})\right)
\end{equation}
for $(x,x^{\prime}) \in \mathbb{R}^d \times \mathbb{R}^{d^{\prime}}$ and $(b,b^{\prime}) \in B \times B^{\prime}.$ Note that the measure $\mu\times\mu'$ is the invariant measure for the  $\{\Upsilon_{(b,b')}\}_{(b,b')\in B\times B'}$.

Define now the filters 
\begin{equation}
m_{(b,b^{\prime})}(x,x^{\prime})=e^{2 \pi i l(b,b^{\prime}) \cdot x} H_{(b,b^{\prime})} (x,x^{\prime}),
\end{equation}
for $(x,x^{\prime}) \in \mathbb{R}^d \times \mathbb{R}^{d^{\prime}}, (b,b^{\prime}) \in B \times B^{\prime},$ where 
$$H_{(b,b^{\prime})}(x,x^{\prime})=\sum\limits_{(c,c^{\prime}) \in B \times B^{\prime}} a_{(b,b^{\prime}),(c,c^{\prime})} \chi_{\Upsilon_{(c,c^{\prime})}(X_B \times X_{B^{\prime}})}(x,x^{\prime}).$$
($\chi_A$ denotes the characteristic function of the set $A$).

With these filters we define the operators $S_{(b,b^{\prime})}$
on $L^{2}(\mu \times \mu^{\prime})$ by
\begin{equation}
\left( S_{(b,b^{\prime})} f \right)(x,x^{\prime})=m_{(b,b^{\prime})}(x,x^{\prime}) f(\mathcal{R}x,\mathcal{R}'x^{\prime}).
\end{equation}

\begin{lemma}
The operators $S_{(b,b')}$, $(b,b')\in B\times B'$ are a representation of the Cuntz algebra $\mathcal O_{NN'}$, i.e., they satisfy the relations in \eqref{eqcuntz}. The adjoint $S_{(b,b')}^*$ is given by the formula
\begin{equation}
(S_{(b,b')}^*f)(x,x')=\frac{1}{NN'}\sum_{(c,c')\in B\times B'}\cj{m_{(b,b')}}(\Upsilon_{(c,c')}(x,x'))f(\Upsilon_{(c,c')}(x,x')),
\label{eqadj}
\end{equation}
for $f\in L^2(\mu\times\mu'), (x,x')\in X_B\times X_{B'}$.
\end{lemma}

\begin{proof}
First, we compute the adjoint, using the invariance equations for $\mu\times\mu'$.
$$\ip{S_{(b,b')}f}{g}=\int m_{(b,b')}(x,x')f(\mathcal R x,\mathcal R' x')\cj g(x,x')\,d(\mu\times\mu')$$$$=\frac{1}{NN'}\sum_{(c,c')}\int m_{(b,b')}(\Upsilon_{(c,c')}(x,x'))f(x,x')\cj g(\Upsilon_{(c,c')}(x,x'))\,d(\mu\times\mu'),$$
and this proves \eqref{eqadj}.

A simple computation shows that the Cuntz relations are equivalent to the following matrix being unitary for all $(x,x^{\prime}) \in X_B\times X_{B'}$:
\begin{equation}
\frac{1}{\sqrt{N N^{\prime}}} \left( m_{(b,b^{\prime})}  \left( \Upsilon_{(c,c^{\prime})} (x,x^{\prime}) \right)  \right)_{(b,b^{\prime}),(c,c^{\prime}) \in  B \times B^{\prime}}.
\end{equation}
This means that 
$$\frac{1}{\sqrt{N N^{\prime}}} \left(e^{2 \pi i l(b,b^{\prime}) \cdot {{R}}^{-1}(x+c)} a_{(b,b^{\prime}),(c,c^{\prime})} \right)_{(b,b^{\prime}),(c,c^{\prime}) \in  B \times B^{\prime}}.$$
should be unitary for all $x \in \mathbb{R}^d.$
Equivalently, for all  $(c_1,c_1^{\prime}), (c_2,c_2^{\prime}) \in B \times B^{\prime}$,
\begin{equation*}
\begin{split}
\delta_{(c_1, c_1^{\prime}), (c_2,c_2^{\prime})} &= \frac{1}{N N^{\prime}} \sum\limits_{(b,b^{\prime}) \in B \times B^{\prime}} e^{2 \pi i l(b,b^{\prime}) \cdot R^{-1}(x+c_{1})} a_{(b,b^{\prime}),(c_{1},c_1^{\prime})} \cdot e^{-2 \pi i l(b,b^{\prime}) \cdot R^{-1}(x+c_{2})} \cj{a}_{(b,b^{\prime}),(c_{2},c_2^{\prime})}\\
&=\frac{1}{N N^{\prime}} \sum\limits_{(b,b^{\prime}) \in B \times B^{\prime}} e^{2 \pi i l(b,b^{\prime}) \cdot R^{-1}c_{1}} a_{(b,b^{\prime}),(c_{1},c_1^{\prime})} \cdot e^{-2 \pi i l(b,b^{\prime}) \cdot R^{-1}c_2} a_{(b,b^{\prime}),(c_{2},c_2^{\prime})}.
\end{split}
\end{equation*}
which is true, by \eqref{eqp1.1}.

\end{proof}

For a word $\omega=(b_1,b_1^{\prime}) . . . (b_k,{b_k}^{\prime})$ we compute
$$(S_{\omega}\mathbf{1})(x,x') =(S_{(b_1,b_1^{\prime})} . . . S_{(b_k,b_k^{\prime})} \mathbf{1})(x,x^{\prime}) 
=S_{(b_1,{b_1}^{\prime})} . . . S_{(b_{k-1},{b_{k-1}^{\prime})}} e^{2 \pi i  l (b_k,{b^{\prime}}_k) \cdot x} H_{(b_k,b_k^{\prime})}(x,x^{\prime})$$
$$=S_{(b_1,b_1^{\prime})} . . . S_{(b_{k-2},b_{k-2}^{\prime})} e^{2 \pi i  l (b_{k-1},{b^{\prime}}_{k-1}) \cdot x} \cdot e^{2 \pi i  l (b_{k},b_k^{\prime}) \cdot \mathcal{R}x}H_{(b_{k-1},b_{k-1}^{\prime})}(x,x^{\prime}) H_{(b_{k},b_k^{\prime})}(\mathcal{R}x,\mathcal{R}^{\prime}x^{\prime})
=\dots=$$$$e^{2 \pi i \left( l(b_{1},b_1^{\prime}) \cdot x + l(b_{2},b_2^{\prime}) \cdot \mathcal{R}x + . . . + l(b_{k},b_k^{\prime}) \cdot \mathcal R^{k-1}x  \right)} H_{(b_{1},b_1^{\prime})}(x,x^{\prime}) H_{(b_{2},b_2^{\prime})}(\mathcal{R}x,\mathcal{R}^{\prime}x^{\prime}) \ . \ . \ . \ H_{(b_{k-1},b_{k-1}^{\prime})}(\mathcal{R}^{k-1}x,{\mathcal{R}^{\prime}}^{k-1}x^{\prime}) .$$

Since $l(b_1,b_1^{\prime}) \cdot {R}^{k}b_2 \in \mathbb{Z}$ for all $(b,b^{\prime}) \in  B \times B^{\prime}, k \geq 0, b_2 \in B,$ we get that the first term in the above product is 
$$
e^{2 \pi i \left( l(b_1,b_1^{\prime})+{R}^{T}l(b_2,b_2^{\prime}) + . . . + ({R}^{T})^{k-1} l(b_k,b_k^{\prime})  \right) \cdot x}.
$$

Next we will compute the projection $P_{V}S_{\omega} \mathbf{1},$ onto the the subspace $$V=\{ f(x,y)=g(x) : g \in L^2(\mu) \}.$$ It is easy to see that the projection onto $V$ is given by the formula
$$(P_Vf)(x)=\int f(x,x')\,d\mu'(x'),\quad (f\in L^2(\mu\times\mu')).$$

For this we compute:
$$
A:=\int H_{(b_1,b_1^{\prime})}(x,x^{\prime}) . \ . \ . \ H_{(b_{k},b_k^{\prime})}(\mathcal{R}^{k-1}x,{\mathcal{R}^{\prime}}^{k-1}x^{\prime}) \ d{\mu}^{\prime}(x^{\prime})
=( \textnormal{ using the invariance equation for} \ \  \mu^{\prime}) $$$$
=\frac{1}{N^{\prime}} \sum\limits_{c^{\prime} \in B^{\prime}} \int H_{(b_1,b_1^{\prime})}(x,\tau'_{c'}x^{\prime}) . \ . \ . \ H_{(b_{k},b_k^{\prime})}(\mathcal{R}^{k-1}x,{\mathcal{R}^{\prime}}^{k-1}\tau'_{c'}x^{\prime})  d{\mu}^{\prime}(x^{\prime}).$$
We have 
$$\frac{1}{N^{\prime}} \sum\limits_{c^{\prime} \in B^{\prime}} H_{(b_1,b_1^{\prime})}(x,\tau'_{c'}x')=\frac{1}{N^{\prime}} \sum\limits_{c^{\prime} \in B^{\prime}} a_{(b_1,b_1^{\prime}), (b(x),c^{\prime})}=\alpha_{(b,b')}$$
where $b(x)=b$ if $x \in \tau_b(X_B).$ 

Then we obtain further, by induction 

$$A=\alpha_{(b_1,b_1')}\int  H_{(b_2,b_2^{\prime})}(\mathcal{R} x,x^{\prime}) \dots H_{(b_k,b_k^{\prime})} (\mathcal{R}^{k-1}x,{\mathcal{R}^{\prime}}^{k-2}x') \  d{\mu}^{\prime}(x^{\prime}) = \dots=\prod\limits_{j=1}^{k} \alpha_{(b_j,b_j')}.$$

So
\begin{equation}
P_{V}S_{\omega}\mathbf{1}(x,x^{\prime})=e^{2 \pi i \left( l(b_1,b_1^{\prime}) + {R}^{T}l(b_2,b_2^{\prime})+ \dots + {({R}^{T}})^{k-1} l(b_k,b_k^{\prime}) \right) \cdot x} \cdot \prod\limits_{j=1}^{k} \alpha_{(b_j,b_j')}. 
\label{eq1.18}
\end{equation}

\begin{remark}\label{rem2.2}
Note that in order for the functions $P_VS_\omega\mathbf 1$ to be just some weighted exponential functions the restriction \eqref{eqp1.2} that we have on the numbers $a_{(b,b'),(c,c')}$ is necessary; since we want  
this function to be just a weighted exponential function, we should have that
$$\frac{1}{N^{\prime}} \sum\limits_{c^{\prime} \in B^{\prime}} a_{(b , b^{\prime}), (b(x),c^{\prime})}$$
is independent of $x,$ which means that 
\begin{equation}
\frac{1}{N^{\prime}} \sum\limits_{c^{\prime} \in B^{\prime}} a_{(b , b^{\prime}), (c,c^{\prime})}=\alpha_{(b , b^{\prime})}
\end{equation}
should be independent of $c$, and this is guaranteed by \eqref{eqp1.2}. So it is not enough to make the matrix in \eqref{eqp1.1} unitary, we need also the condition in \eqref{eqp1.2}. Equation \eqref{eqp1.1} implies the Cuntz relations, equation \eqref{eqp1.2} guarantees that the projections are weighted exponential functions. 
\end{remark}

We compute now  for $e_t(x)=e_{(t,0)}(x,x')$: 
$$
\left( {S^*}_{(b,b^{\prime})} (e_{t}) (x,x^{\prime})  \right) = \frac{1}{N N^{\prime}} \sum\limits_{(c,c^{\prime}) \in B \times B^{\prime}} \cj{m}_{(b,b^{\prime})} (\tau_c x , \tau'_{c'} x^{\prime})  e_{t} (\tau_c x ) $$$$
=\frac{1}{NN^{\prime}} \sum\limits_{(c,c^{\prime}) \in B \times B^{\prime}} e^{-2 \pi i  l(b,b^{\prime}) \cdot {{R}}^{-1} (x+c)} \cj{a}_{(b,b^{\prime}) , (c,c^{\prime})}  e^{2 \pi i t \cdot {{R}}^{-1} (x+c)} $$$$
=\frac{1}{N N^{\prime}} \sum\limits_{(c,c^{\prime})} e^{2 \pi i(t - l(b,b^{\prime})) \cdot {{R}}^{-1}c} \cj{a}_{(b,b^{\prime}) , (c,c^{\prime})} \cdot e^{2 \pi i (t- l(b,b^{\prime})) \cdot {{R}}^{-1}x} $$
$$=\left( \frac{1}{N} \sum\limits_{c \in B} \left( \frac{1}{N^{\prime}} \sum\limits_{c^{\prime} \in B^{\prime}} \cj{a}_{(b,b^{\prime}),(c,c^{\prime})}  \right)  e^{2 \pi i ({{R}}^{T})^{-1} (t-l(b,b^{\prime})) \cdot c} \right) \cdot e_{({{R}}^{T})^{-1} (t-l(b,b^{\prime}))}(x)$$
$$=\left( \frac{1}{N} \sum\limits_{c \in B} \cj{\alpha}_{(b,b^{\prime})} \cdot e^{2 \pi i ({{R}}^{T})^{-1} (t-l(b,b^{\prime})) \cdot c}  \right) \cdot e_{({{R}}^{T})^{-1} (t-l(b,b^{\prime}))}(x).$$

So
\begin{equation}
{S^{*}}_{(b,b^{\prime})} e_t =\cj{\alpha}_{(b,b^{\prime})} m_B\left( ({R}^{T})^{-1} (t-l(b,b^{\prime}))\right)  e_{({{R}}^{T})^{-1} (t-l(b,b^{\prime}))}.
\end{equation}

Define 

$$g_{(b,b^{\prime})}(x)=({{R}}^{T})^{-1} (x-l(b,b^{\prime})) \ \   ( (b,b^{\prime}) \in B \times B^{\prime} ).$$
Then
\begin{equation}
{S^{*}}_{(b,b^{\prime})} e_t = \cj{\alpha}_{(b,b^{\prime})} {m}_{B}({g_{(b,b^{\prime})}} (t))e_{g_{(b,b')}(t)}.
\label{eq1.12}
\end{equation}

Define $\Omega(B\times B')$ to be the empty word and all the finite words $\omega=(b_0,b_0')\dots (b_k,b_k')$ in $(B\times B')^k$, $k\in\bn$ that do not end in $(0,0)$, i.e., $(b_k,b_k')\neq(0,0)$. It is easy to see, using the Cuntz relations and the fact that $S_{(0,0)}\mathbf 1=\mathbf 1$, that the family of  vectors $\mathcal E:=\{S_\omega\mathbf 1 : \omega\in \Omega(B\times B')\}$ is orthonormal. Note also that 
$$\mathcal E=\cup_{(b,b')\in B\times B'}S_{(b,b')}\mathcal E,$$
and the union is disjoint. 

We want to prove that the subspace $V$ is contained in the closed span of $\{S_\omega\mathbf 1 : \omega\in \Omega(B\times B')\}$, which we denote by $\mathcal K$. 

Using the Cuntz relations and \eqref{eq1.12}, we have 
$$
1=\| e_t \|^2 =\sum\limits_{(b,b^{\prime}) \in B \times B^{\prime}} \langle\ S_{(b,b^{\prime})} {S_{(b,b^{\prime})}^{*}}e_t, e_t \rangle 
=\sum\limits_{(b,b^{\prime}) \in B \times B^{\prime}} \| {S_{(b,b^{\prime})}^{*}}e_t \|^2 $$$$ 
=\sum\limits_{(b,b^{\prime}) \in B \times B^{\prime}} | \alpha_{(b,b^{\prime})} |^{2} |m_{B} ( {g_{(b,b^{\prime})}} (t) ) |^2,
$$
So
\begin{equation}
1=\sum\limits_{(b,b^{\prime}) \in B \times B^{\prime}} | \alpha_{(b,b^{\prime})} |^{2} |m_{B} ( {g_{(b,b^{\prime})}}(t) ) |^2.
\label{eq1.22}
\end{equation}

Define the function $f:\br^d\rightarrow V$ by $f(t)=e_t$, where $e_t(x,x')=e^{2\pi i x\cdot t}$. Note that $f(0)=\mathbf 1\in\mathcal K$. Also, we define the function $h:\br^d\rightarrow\br$ by
\begin{equation}
h(t)=\sum_{\omega\in\Omega(B\times B')}|\ip{f(t)}{S_\omega\mathbf 1}|^2=\|P_{\mathcal K}f(t)\|^2.
\label{eqh}
\end{equation}

Note that $0\leq h\leq 1$, $h(0)=1$ and $h$ can be extended to $\bc^d$ by
$$h(z)=\sum_{\omega\in\Omega(B\times B')}|\ip{e_z}{S_\omega\mathbf 1}|^2=\|P_{\mathcal K}e_z\|^2,$$
where, for $z\in \bc^d$, $e_z(x,x')=e^{2\pi i x\cdot z}$. We will prove that $h(z)$ is an entire function.

For a fixed $\omega\in\Omega(B\times B')$, define $f_\omega:\bc^d\rightarrow\bc$ by
$$f_\omega(z)=\int e^{2\pi i z\cdot x}\cj{S_\omega\mathbf 1(x,x')}\,d\mu(x)\,d\mu'(x').$$
Since the measure $\mu\times\mu'$ is compactly supported, a standard convergence argument shows that the function $f_\omega$ is entire. Similarly, $f_\omega^*(z):=\cj{f_\omega(\cj z)}$ is entire, and for real $t$,
$$f_\omega(t)f_\omega^*(t)=\ip{e_t}{S_\omega\mathbf 1}\cdot\cj{\ip{e_t}{S_\omega\mathbf 1}}=|\ip{e_t}{S_\omega\mathbf 1}|^2.$$
Thus,
$h(t)=\sum_{\omega\in\Omega(B\times B')}f_\omega(t)f_\omega^*(t).$
For $n\in\bn$, let $h_n(z)=\sum_{|\omega|\leq n}f_{\omega}(z)f_\omega^*(z)$, which is entire. By H\" older's inequality
$$\sum_{\omega\in\Omega(B\times B')}|f_\omega(z)f_\omega^*(z)|\leq\left(\sum_{\omega}|\ip{e_z}{S_\omega\mathbf 1}|^2\right)^{\frac12}\left(\sum_{\omega}|\ip{e_{\cj z}}{S_\omega\mathbf 1}|^2\right)^{\frac12}$$
$$\leq \|e_z\|\|e_{\cj z}\|\leq e^{K\sqrt{(\Im z_1)^2+\dots+(\Im z_d)^2}},$$
for some constant $K$. Thus, the sequence $h_n(z)$ converges pointwise to the function $h(z)$, and is uniformly bounded on bounded sets. By the theorems of Montel and Vitali, the limit function is entire. 

We prove that $h$ satisfies \eqref{eqt1.3}:
$$h(t)=\sum_{v\in \mathcal E}\left|\ip{v}{e_t}\right|^2=\sum_{(b,b')\in B\times B'}\sum_{v\in\mathcal E}\left|\ip{S_{(b,b')}v}{e_t}\right|^2=\sum_{(b,b')\in B\times B'}\sum_{v\in\mathcal E}\left|\ip{v}{S_{(b,b')}^*e_t}\right|^2$$
$$=\sum_{(b,b')}\sum_{v}\left|\ip{v}{\alpha_{(b,b')}m_B(g_{(b,b')}(t))e_{g_{(b,b')}(t)}}\right|^2=
\sum_{(b,b')}|\alpha_{(b,b')}|^2|m_B(g_{(b,b')}(t))|^2h(g_{(b,b')}(t)).
$$

By \eqref{eqt1.5}, we get that $h$ has to be constant, and since $h(0)=1$ we obtain that $h\equiv 1$. But this implies that $\|P_{\mathcal K}e_t\|=1=\|e_t\|$ and therefore $e_t$ is in $\mathcal K$ for all $t\in\br^d$. Since, by the Stone-Weierstrass theorem the functions $e_t$ span $V$, we get that $V$ is a subspace of $\mathcal K$. Then, with Lemma \ref{lemframe}, we obtain that $P_VS_\omega\mathbf1$, $\omega\in\Omega(B\times B')$ forms a Parseval frame for $V$. Going back to the index set $L$ and discarding the zeroes, we obtain the conclusion.

\end{proof}

\section{Further remarks}

\begin{remark}\label{remcomp}
Suppose that the conditions (i) and (ii) in Theorem \ref{th1} are satisfied. Suppose in addition that the set in \eqref{eqt1.5} spans $L^2(\mu)$. Then this set is a Parseval frame.

Indeed, if we follow the proof of Theorem \ref{th1}, we see that the function $h$ in \eqref{eqh} is the constant 1, and that is what was needed in the proof. 
\end{remark}

\begin{corollary}\label{corbigger}
Suppose we have a set $L$ and some complex numbers $\{\alpha_l\}$, such that the conditions (i)--(iii) in Theorem \ref{th1} are satisfied. Suppose now that $L'$ is a set that contains the set $L$ and $\{\alpha_l' :l\in L'\}$ are some nonzero numbers, such that the conditions (i) and (ii) are satisfied. Then the conclusion of Theorem \ref{th1} holds for $L'$ and $\{\alpha_l'\}$. 
\end{corollary}

\begin{proof}
Theorem \ref{th1} shows that $\{e_{l_0+R^Tl_1+\dots +(R^T)^kl_k} :l_0\dots l_k\in\Omega(L)\}$ spans $L^2(\mu)$. But then $\{e_{l_0+R^Tl_1+\dots +(R^T)^kl_k} :l_0\dots l_k\in\Omega(L')\}$ is a bigger set so it also spans $L^2(\mu)$. The rest follows from Remark \ref{remcomp}.
\end{proof}

\begin{proposition}\label{pr1.10}
With the notations in Theorem \ref{th1}, suppose the conditions (i) and (ii) are satisfied. Then for all $l\neq 0$ in $L$,
\begin{equation}
\alpha_{l}=0 \ \textnormal{or} \ \sum\limits_{b \in B} e^{2 \pi i (R^T)^{-1}l \cdot b}=0.
\label{eq1.10.1}
\end{equation}
In particular $L$ cannot contain numbers $l\neq 0$ with $l\cdot R^{-1}b\in\bz$ for all $b\in B$, which implies also that if $l\neq 0$ is in $L$, then $l\not\equiv 0(\mod R^T\bz)$. 
\end{proposition}

\begin{proof}
Using the proof of Theorem \ref{th1}, since the matrix in \eqref{eqp1.1} is unitary and the first row is constant we have that 
\begin{equation*}
\begin{split}
0&=\sum\limits_{(c,c^{\prime}) \in B \times B^{\prime}} e^{2 \pi i (R^T)^{-1}l(b,b^{\prime}) \cdot c} a_{(b,b^{\prime}),(c,c^{\prime})} \cdot {1} \\
&=\sum\limits_{c \in B} e^{2 \pi i (R^T)^{-1}l(b,b^{\prime}) \cdot c} \sum\limits_{c^{\prime} \in B^{\prime}} a_{(b,b^{\prime}),(c,c^{\prime})} \\ 
&=\sum\limits_{c \in B} e^{2 \pi i (R^T)^{-1}l(b,b^{\prime}) \cdot c} \cdot N^{\prime} \alpha_{(b,b^{\prime})}.
\end{split}
\end{equation*}

\end{proof}

\begin{remark}\label{rem1.11}
The condition (ii) in Theorem \ref{th1} is necessary for this construction to work. 
Suppose the numbers $a_{(b,b'),(c,c')}$ and $l(b,b')$ satisfy \eqref{eqp1.1} and \eqref{eqp1.2} (see also Remark \ref{rem2.2}). Equation \eqref{eqp1.1} is necessary for the Cuntz relations, and equation \eqref{eqp1.2} is necessary for the projections of the functions $S_\omega\mathbf 1$ to be weighted exponential functions. Then the matrix 
$$\frac{1}{\sqrt{N}}\left(e^{2\pi i (R^T)^{-1}l(b,b')\cdot c}\alpha_{(b,b')}\right)_{(b,b')\in B\times B', c\in B}$$
is an isometry, as in \eqref{eqt1.1}.

\begin{proof}

From \eqref{eq1.22} we get 
\begin{equation*}
\begin{split}
1&=\sum\limits_{(b,b^{\prime}) \in B \times B^{\prime}} | \alpha_{(b,b^{\prime})} |^{2}  \frac{1}{N^2} \sum\limits_{c_1,c_2\in B} e^{2 \pi i ( {{R}}^{T}  )^{-1}  (t - l(b,b^{\prime})) \cdot (c_1 - c_2)   } \\ 
&=\sum\limits_{k \in B-B} e^{2 \pi i ( {{R}}^{T} )^{-1} t \cdot k   } \frac{1}{N^2} \sum\limits_{\substack {c_1 , c_2 \in B \\ c_1-c_2=k}} \sum\limits_{(b,b^{\prime})\in B\times B'} e^{-2 \pi i ( {{R}}^{T}  )^{-1}  l(b,b^{\prime}) \cdot k} |\alpha_{(b,b^{\prime})}|^2.
\end{split}
\end{equation*}
We have two trigonometric polynomials which are equal to each other and therefore their coefficients must be equal. \\
For $k=0$ we get 
$$1=\frac{1}{N} \sum\limits_{(b,b^{\prime})} |\alpha_{(b,b^{\prime})}|^2.$$
For $k \neq 0$ we get 
$$0=\frac{1}{N^2} \#{\{ (c_1 , c_2) \in B \times B : c_1 - c_2 =k \}} \sum\limits_{(b,b^{\prime})} e^{-2 \pi i (R^T)^{-1}l(b,b^{\prime}) \cdot k} |\alpha_{(b,b^{\prime})}|^2,$$
which implies \eqref{eqt1.1}.
\end{proof}

Also, the condition $\alpha_0=0$ is necessary for this construction. We want $S_0\mathbf 1=\mathbf 1$, to make sure that the set $\mathcal E=\{S_\omega \mathbf 1 : \omega\in \Omega(B\times B')\}$ is invariant as in Theorem \ref{thoc}(i), so the numbers $a_{(0,0),(c,c')}$ should all be $1$. Therefore 
$$\alpha_0=\frac{1}{N'}\sum_{c'\in B' }a_{(0,0),(c,c')}=1.$$
\end{remark}

\begin{remark}\label{remc3}
This construction cannot be used for the Middle Third Cantor set. In that example, $R=3$, $B=\{0,2\}$. From Proposition \ref{pr1.10}, we see that we must have $1+e^{2\pi i \cdot \frac{2l}3}=0$ for some integer $l$, and that is impossible. 
\end{remark}

\begin{proposition}\label{preq}
With the notation of Theorem \ref{th1}, the following statements are equivalent. 
\begin{enumerate}
	\item The matrix $$	T=\frac{1}{\sqrt{N}} \left(  e^{2 \pi i {({R}^{T}})^{-1} l \cdot b} \alpha_{l} \right)_{l\in L, b\in B}$$

has orthonormal columns, i.e., it is an isometry.
\item For every $t\in \br^d$
\begin{equation}
\sum_{l\in L}|\alpha_l|^2\left|m_B((R^T)^{-1}(t-l))\right|^2=1.
\label{eq1.10.2}
\end{equation}
\item The functions $\{\alpha_le_l : l\in L\}$  form a Parseval frame for the space $L^2(\delta_{R^{-1}B})$, where $\delta_{R^{-1}B}$ is the measure $\delta_{R^{-1}B}=\frac{1}{N}\sum_{b\in B}\delta_{R^{-1}b}$ ($\delta_a$ denotes the Dirac measure at $a$). 
\end{enumerate}
\end{proposition}

\begin{proof}
The condition in (ii) can be rewritten as
$$1=\frac{1}{N^2}\sum_{l\in L}|\alpha_l|^2\sum_{c,c'\in B}e^{2\pi i (R^T)^{-1}(t-l)\cdot(c-c')}=\frac{1}{N^2}\sum_{k\in B-B}e^{2\pi i (R^T)^{-1}t\cdot k}\sum_{\substack{c,c'\in B\\ c-c'=k}}\sum_{l\in L}|\alpha_l|^2e^{2\pi i (R^T)^{-1}l\cdot k}$$
$$=\frac{1}{N^2}\sum_{k\in B-B}e^{2\pi i (R^T)^{-1}t\cdot k}\#\{(c,c')\in B\times B: c-c'=k\}\sum_{l\in L}|\alpha_l|^2e^{2\pi i (R^T)^{-1}l\cdot k}.$$
Equating the coefficients, this is equivalent to 
$$\frac{1}{N}\sum_{l\in L}|\alpha_l|^2=1,$$
and
$$\frac{1}{N}\sum_{l\in L}|\alpha_l|^2e^{2\pi i (R^T)^{-1}l\cdot (c-c')}=0,$$
for all $c\neq c'$ in $B$. Thus (i) and (ii) are equivalent. 

The condition in (iii) can be rewritten as: for all $u,v\in L^2(\delta_B)$, with $t_l=\alpha_le_l$ in $L^2(\delta_{R^{-1}B})$,
$$\sum_{c\in B}u_c\cj v_c=\ip{u}{v}=\sum_{l\in L}\ip{u}{t_l}\ip{t_l}{v}=\sum_{l\in L}\sum_{c,c'\in B}u_c\cj v_{c'}|\alpha_l|^2e^{-2\pi il\cdot R^{-1}(c-c')}$$
$$=\sum_{c,c'\in B}u_c\cj v_{c'}\sum_{l\in L}|\alpha_l|^2e^{-2\pi i (R^T)^{-1}l\cdot (c-c')}.$$
Using the canonical vectors $u:=\delta_{b}$, $v=\delta_{b'}$ we obtain that the statements in (i) and (iii) are equivalent. 
\end{proof}

\begin{definition}\label{defru}
We define the {\it Ruelle (transfer) operator} for functions $f$ defined on $\br^d$ (or just on $X_B$) by
$$R f(t)=\sum_{l\in L}|\alpha_l|^2|m_B(g_l(t))|^2f(g_l(t)),$$
where
$$g_l(t)=(R^T)^{-1}(t-l),\quad (l\in L,t\in \br^d).$$

Note that the condition \eqref{eqt1.3} can be rewritten as $Rh=h$ and condition (ii) in Proposition \ref{preq} means that $R\mathbf 1=\mathbf 1$. 
\end{definition}

\begin{proposition}\label{prit}
If one of the equivalent conditions in Proposition \ref{preq} is satisfied then, for any $k\in\bn$, the functions 
$$\left\{\alpha_{l_0}\dots\alpha_{l_{k-1}}e_{l_0+R^Tl_1+\dots+(R^T)^{k-1}l_{k-1}} : l_i\in L\right\}$$
form a Parseval frame for the space $L^2(\delta_{R^{-k}(B+RB+\dots+R^{k-1}B)})$. 
\end{proposition}

\begin{proof}
Since $R^k\mathbf 1=\mathbf 1$, we obtain by induction 
$$1=\sum_{l_0,\dots,l_{k-1}\in L}|\alpha_{l_0}|^2\dots|\alpha_{l_{k-1}}|^2|m_B(g_{l_{k-1}}t)|^2|m_B(g_{l_{k-2}}g_{l_{k-1}}t)|^2\dots|m_B(g_{l_{0}}\dots g_{l_{k-1}}t)|^2$$
$$=\sum_{l_0,\dotsm l_{k-1}\in L}|\alpha_{l_0}|^2\dots|\alpha_{l_{k-1}}|^2\left|m_{B+RB+\dots +R^{k-1}B}\left((R^T)^{-k}(t- (l_0+R^Tl_1+\dots +(R^T)^{k-1}l_{k-1})\right)\right|^2,$$
where 
$$m_{B+RB+\dots +R^{k-1}B}:=\frac{1}{N^k}\sum_{b_0,\dots,b_{k-1}\in B}e^{2\pi i x\cdot (b_0+Rb_1+\dots+R^{k-1}b_{k-1})}=m_B(x)m_B((R^T)x)\dots m_B((R^T)^{k-1}x).$$
The proposition follows now from the equivalence between (ii) and (iii) in Proposition \ref{preq}.
\end{proof}

\begin{definition}\label{defcong}
For $k\in\bz^d$, we denote 
$$[k]:=\{k'\in\bz : (k'-k)\cdot R^{-1}b\in\bz, \mbox{ for all } b\in B\}.$$
We denote by $[L]:=\{[l] : l\in L\}$. 
\end{definition}

\begin{proposition}\label{pr1.13}
Assume $\alpha_l\neq 0$ for all $l\in L$. 
If one of the equivalent conditions in Proposition \ref{preq} is satisfied then:
\begin{enumerate}
	\item $N\leq \#[L]$.
	\item For all $l_0\in L$,
\begin{equation}
\sum_{{l\in L, l\in [l_0]}}|\alpha_l|^2\leq1.
\label{eq[l]}
\end{equation}
\end{enumerate} 
\end{proposition}

\begin{proof}
Since $\{\alpha_l e_l : l\in L\}$ form a Parseval frame for $L^2(\delta_{R^{-1}B})$, we have 
$$|\alpha_{l_0}|^2=\|\alpha_{l_0}e_{l_0}\|^2=\sum_{l\in L}|\ip{\alpha_{l_0}e_{l_0}}{\alpha_le_l}|^2.$$
Then, since for $l\in [l_0]$ we have $e_l=e_{l_0}$ in $L^2(\delta_{R^{-1}B})$, we get that
$$\sum_{l\in L, l\in[l_0]}|\alpha_{l_0}|^2|\alpha_l|^2\leq |\alpha_{l_0}|^2.$$
This implies that 
$$\sum_{{l\in L, l\in [l_0]}}|\alpha_l|^2\leq 1.$$

Also, 
$$N=\sum_{l\in L}|\alpha_l|^2=\sum_{[l_0]\in [L]}\sum_{l\in L\cap[l_0]}|\alpha_l|^2\leq \sum_{[l_0]\in [L]}1=\#[L].$$
\end{proof}

\section{Invariant sets}

\begin{definition}\label{def1}
With the notations of Theorem \ref{th1}, a set $M\subset \br^d$ is called {\it invariant} if for any point $t\in M$, and any $l\in L$, if $\alpha_{l}m_B((R^T)^{-1}(t-l))\neq 0$, then $g_{l}(t):=(R^T)^{-1}(t-l)\in M$. $M$ is said to be non-trivial if $M\neq \{0\}$.

Note that 
$$\sum\limits_{l\in L} | \alpha_{l}  |^2 \ |m_{B}(g_{l} (t))|^{2}=1 \quad (t \in \mathbb{R}^d ),$$
see \eqref{eq1.10.2}, and therefore, we can interpret the number $| \alpha_{l}  |^2 \ |m_{B}(g_{l} (t))|^{2}$ as the probability of transition from $t$ to $g_{l}(t)$, and if this number is not zero then we say that this {\it transition is possible in one step (with digit $l$)}, and we write $t\rightarrow g_{l}(t)$ or $t\stackrel{l}{\rightarrow}g_l(t)$. We say that the {\it transition is possible} from a point $t$ to a point $t'$ if there exist $t_0=t$, $t_1,\dots, t_n=t'$ such that $t=t_0\rightarrow t_1\rightarrow\dots\rightarrow t_n=t'$. The {\it trajectory} of a point $t$ is the set of all points $t'$ (including the point $t$) such that the transition is possible from $t$ to $t'$.

A {\it cycle} is a finite set $\{t_0,\dots,t_{p-1}\}$ such that there exist $l_0,\dots, l_{p-1}$ in $L$ such that $g_{l_0}(t_0)=t_1,\dots, g_{l_{p-1}}(t_{p-1})=t_{p}:=t_0$. Points in a cycle are called {\it cycle points}. 

A cycle $\{t_0,\dots, t_{p-1}\}$ is called {\it extreme} if $|m_B(t_i)|=1$ for all $i$; by the triangle inequality, since $0\in B$, this is equivalent to $t_i\cdot b\in \bz$ for all $b\in B$. 

\end{definition}

The next proposition gives some information about the structure of finite, minimal sets, which makes it easier to find such sets in concrete examples. 
\begin{proposition}\label{pr3}
Assume $\alpha_l\neq 0$ for all $l\in L$. 
Let $M$ be a non-trivial finite, minimal invariant set. Then, for every two points $t,t'\in M$ the transition is possible from $t$ to $t'$ in several steps. In particular, every point in the set $M$ is a cycle point. In dimension $d=1$, $M$ is contained in the interval $\left[\frac{\min(-L)}{R-1},\frac{\max(-L)}{R-1}\right]$. 

If $t$ is in $M$ and if there are two possible transitions $t\rightarrow g_{l_1}(t)$ and $t\rightarrow g_{l_2}(t)$, then $l_1\equiv l_2(\mod R^T\bz^d)$.

Every point $t$ in $M$ is an extreme cycle point, i.e., $|m_B(t)|=1$ and if $t\rightarrow g_{l_0}(t)$ is a possible transition in one step, then $[l_0]\cap L=\{l\in L : l\equiv l_0(\mod (R^T)\bz^d)\}$ (with the notation in Definition \ref{defcong}) and
\begin{equation}
\sum_{l\in L,l\equiv l_0(\mod R^T\bz^d)}|\alpha_l|^2=1.
\label{eqac}
\end{equation}

In particular $t\cdot b\in\bz$ for all $b\in B$. 
\end{proposition}

\begin{proof}
Let $t,t'$ be two points in $M$. Suppose the transition from $t$ to $t'$ is not possible. Let $T$ be the trajectory of $t$. Then $T$ is an invariant set contained in $M$ which does not contain $t'$. This contradicts the minimality of $T$.

Take a point $t$ in $M$. Then at least one transition in one step is possible $t\rightarrow g_{l}(t)$, because of \eqref{eq1.10.2}. But then the transition back to $t$ from $g_{l}(t)$ is also possible, so $t$ is a cycle point. Writing the equations for a cycle point (see also the equations for $x_0$ below), we get that, in dimension $d=1$, any cycle point is contained in the interval $\left[\frac{\min(-L)}{R-1},\frac{\max(-L)}{R-1}\right]$. 

If the transitions $t\rightarrow g_{l_0}(t)$ and $t\rightarrow g_{l_0'}(t)$ are possible, then we can close two cycles by some transitions from $g_{l_0}(t)$ and $g_{l_0'}(t)$ back to $t$. We can assume also that these two cycles have the same length, by taking a common multiple of their minimal lengths and going around the cycles several times. If we write the equations for the two cycles, we get that there are some numbers $l_0,\dots, l_{p-1}$, $l_0',\dots,l_{p-1}'$ in $L$ such that 
$$x_0=(R^T)^px_0+(R^T)^{p-1}l_{p-1}+\dots+(R^T)l_1+l_0,$$
$$x_0=(R^T)^px_0+(R^T)^{p-1}l_{p-1}'+\dots+(R^T)l_1'+l_0'.$$
Subtracting, we get that $l_0\equiv l_0'(\mod R^T\bz^d)$.

If $t$ is in $M$ and $t\rightarrow g_{l_0}(t)$ is a possible transition in one step, then for $l\in L$, $l\not\equiv l_0(\mod R^T\bz^d)$ we have $\alpha_lm_B(g_l(t))=0$ (because, from the previous statement, we know that the transition from $t$ with $l$ in not possible in one step). Therefore, with \eqref{eq[l]},
$$1=\sum_{l\in L, l\equiv l_0(\mod R^T\bz^d)}|\alpha_l|^2|m_B(g_l(t))|^2=|m_B(g_{l_0}(t))|^2\sum_{l\equiv l_0}|\alpha_l|^2\leq \sum_{l\in[l_0]}|\alpha_l|^2\leq 1.$$
Therefore we have equalities in all inequalities, so $|m_B(g_l(t))|=1$, $[l_0]\cap L=\{l\in L : l\equiv l_0(\mod R^T\bz^d)\}$ and \eqref{eqac} holds. Since $t$ is on a cycle, it can be written as $t=g_{l_1}(t')$ for a $t'$ in $M$ and therefore $|m_B(t)|=1$. Using the triangle inequality $1\leq |m_B(t)|=1$, so we must have equality, and we obtain that $t\cdot b\in\bz$ for all $b\in B$. 

\end{proof}

\begin{example}
We give an example of a finite minimal invariant set where there are two transitions possible from a point and which is a union of extreme cycles that have a common point. 
Let $R=2$, $B=\{0,1\}$, $L=\{0,1,3\}$, and $(\alpha_{l})_{l\in L}=(1,\frac{1}{\sqrt 2},\frac{1}{\sqrt 2})$. It can be easily checked that the conditions (i) and (ii) in Theorem \ref{th1} are satisfied. Then 
$$m_B(x)=\frac{1+e^{2\pi i x}}2.$$
The set $\{-1,-2\}$ is minimal invariant. The only possible transitions in one step from $-1$ are to $-1$, with digit 1 and to $-2$ with digit 3. The only possible transition in one step from $-2$ is to $-1$ with digit 0. 

\end{example}

\begin{theorem}\label{th2}
In dimension $d=1$, let $|R|\geq 2$ and $B$ be as in Theorem \ref{th1} and assume also that the numbers $\alpha_{l}$, $l\in L$ satisfy the conditions (i) and (ii) in Theorem \ref{th1}. If the only finite minimal invariant set is $\{0\}$, then the condition (iii) is satisfied too and therefore the set
\begin{equation}
\left\{\left( \prod_{i=0}^k\alpha_{l_i}\right) e_{l_0+R l_1+\dots +R^{k}l_k} : l_0\dots l_k\in\Omega(L)\right\}
\label{eq2.1}
\end{equation}
is a Parseval frame for $L^2(\mu(R,B))$. 

Conversely, if there is a finite minimal invariant set $M\neq\{0\}$, then $e_c$ is orthgonal to the set in \eqref{eq2.1} for all $c\in M$ (so the set in \eqref{eq2.1} is incomplete).
\end{theorem}

\begin{proof}
Suppose, there is an entire function $h$ such that $0\leq h\leq 1$ on $\br$, $h(0)=1$ and that satisfies \eqref{eqt1.3}. 
Consider the interval $I:=\left[\frac{\min(-L)}{R-1},\frac{\max(-L)}{R-1}\right]$ and note that this interval is invariant for the maps $g_{l}$, i.e., if $x\in I$ then $g_{l}(x)\in I$. Also $0\in I$ since $0\in L$. 

Define the function $\tilde h (x):=h(x)-\min_{t\in I}h(t)$. Using \eqref{eq1.10.2}, we get that $\tilde h $ satisfies \eqref{eqt1.3}, it is entire, $\tilde h \geq 0$ on $I$; in addition $\tilde h $ has a zero in $I$ (the point where $h$ attains its minimum). 

Let $Z$ be the set of zeroes of $\tilde h $ in $I$. If $Z$ is infinite, then, since $\tilde h $ is entire, we get that $\tilde h =0$ so $h$ is indeed constant. Thus $Z$ has to be finite. Since equation \eqref{eqt1.3} is satisfied by $\tilde h $, we have that, if $t\in Z$, then $t\in I$ and $\tilde h (x)=0$, and then, for any $l$, either $\alpha_{l}m_B(g_{l}(t))=0$ or $\tilde h (g_{l}(t))=0$; also $g_{l}(t)\in I$. This implies that $Z$ is a finite invariant set. Then, it must contain a finite minimal invariant set, and, by hypothesis, it must contain $0$, so $\tilde h (0)=0$, which means $\min_I h(x)=1$. But, $0\leq h\leq 1$ so $h$ is constant 1. 

For the converse, suppose $M\neq\{0\}$ is a minimal finite invariant set. Let $c$ in $M$. We prove that $\widehat\mu(c-(l_0+Rl_1+\dots+R^kl_k))=0$, for all $l_0,\dots,l_k\in L$, which implies that $\ip{e_c}{e_{l_0+Rl_1+\dots+R^kl_k} }=0$. 
Suppose not, pick $l_0,\dots,l_k$ such that $\widehat\mu(c-(l_0+Rl_1+\dots+R^kl_k))\neq 0$ and let $l_{k+1}=l_{k+2}=\dots=0$. 

We have that 
$$\widehat\mu(x)=\prod_{n=1}^\infty m_B(R^{-n}x)$$
where the convergence is uniform on compact sets. We obtain that $$m_B\left(\frac{c-(l_0+Rl_1+\dots+R^kl_k)}{R^n}\right)\neq 0$$ for all $n\in\bn$. This implies that 
$$m_B\left(\frac{c-l_0}{R}\right)\neq 0,m_B\left(\frac{c-l_0-Rl_1}{R^2}\right)\neq 0,\dots, m_B\left(\frac{c-l_0-\dots-R^{n-1}l_{n-1}}{R^n}\right)\neq 0,\dots.$$
But this means that the transitions $c\rightarrow g_{l_0}(c)\rightarrow g_{l_1}g_{l_0}(c)\rightarrow\dots g_{l_n}\dots g_{l_0}(c)\rightarrow\dots$ are possible in one step. Since for $n$ large $l_n=0$, this implies that $g_{l_n}\dots g_{l_0}(c)$ converges to $\{0\}$ and therefore $0\in M$. But $M$ is minimal so $\{0\}=M$, a contradiction. 
\end{proof}

\section{Examples}

\begin{example}
Let $N\in\bn$, $N\geq 3$. Let $R=N$ and $B=\{0,1,2,\dots,N-1\}$. Then the invariant measure $\mu$ is the Lebesgue measure on $[0,1]$. In this example, we will not obtain new Parseval frames for $L^2[0,1]$, because the frequencies are in $\bz$, which already give the well know orthonormal Fourier basis. However, we obtain some interesting properties of the representations of integers in base $N$. Define

Let, $L_1,\dots, L_{N-1}$ be some finite, non-empty sets of integers, with $0\in L_i$, $i=1,\dots,N-2$, and $-1\in L_{N-1}$.
$$L:=\{0\}\cup (1+NL_1)\cup(2+NL_2)\cup\dots\cup ((N-2)+NL_{N-2})\cup ((N-1)+NL_{N-1}).$$
Note that the complete set of representatives $\mod N$, $\{0,1,\dots,N-2,-1\}$ is contained in $L$.

Pick some nonzero complex numbers such that 
$$\sum_{l\in L_i}|\alpha_{i+Nl}|^2=1,\quad(i=1,\dots,N-1).$$

We check that the conditions of Theorem \ref{th2} are satisfied. We have for $b,b'$ in $B$:
$$\sum_{l\in L}|\alpha_l|^2e^{2\pi i R^{-1}l\cdot (b-b')}=1+\sum_{i=1}^{N-1}e^{2\pi i N^{-1}i\cdot(b-b')}\sum_{l\in L_i}|\alpha_{i+Nl}|^2=1+\sum_{i=1}^{N-1}e^{2\pi i N^{-1}i\cdot(b-b')}=N\delta_{bb'}.$$

Now consider a minimal finite invariant set $M$. By Proposition \ref{pr3}, every point $x_0$ in $M$ is an extreme cycle point, so $|m_B(x_0)|=1$ which implies that $x_0=:k_0\in\bz$. We prove that $0\in M$, which, by minimality, implies $M=\{0\}$. Consider the subset $J_0:=\{-1,0,1,\dots,N-2\}$ of $L$. This is a complete set of representatives $\mod N$. Therefore there exists $l_0\in J_0$ and $k_1\in \bz$ such that $k_0=l_0+Rk_1$, which means that $k_1=g_{l_0}(k_0)$. Since $m_B(k_1)=1$, the transition $k_0\rightarrow g_{l_0}(k_0)=k_1$ is possible in one step, so $k_1\in M$. By induction, there exist digits $l_0,l_1,\dots$ in $J_0$ such that $k_{n+1}:=g_{l_n}\dots g_{l_0}(k_0)\in\bz\cap M$. Letting $n\rightarrow \infty$, we get that $k_{n+1}\in \left[ \frac{-(N-2)}{N-1}-\epsilon,\frac{1}{N-1}+\epsilon\right]$ for $n$ large enough. Therefore $k_{n+1}=0$ so $0\in M$, and hence $M=\{0\}$. 

Theorem \ref{th2} implies that the set 
$$\{e_{l_0+Rl_1+\dots +R^kl_k} : l_0\dots l_k\in\Omega(L)\}$$
is a Parseval frame for $L^2[0,1]$.

But, the numbers $l_0+Rl_1+\dots+R^kl_k$ are integers and $\{e_n: n\in\bz\}$ is an orthonormal basis for $L^2[0,1]$. So, if we group the elements in the Parseval frame that correspond to a fixed exponential function $e_n$, this implies that for every $n\in\bz$:
\begin{equation}
\sum_{\substack{l_0\dots l_k\in\Omega(L)\\ l_0+Rl_1+\dots +R^kl_k=n}}|\alpha_{l_0}\dots\alpha_{l_k}|^2=1.
\label{eqdec}
\end{equation}

The relation \eqref{eqdec} gives us some information about the ways in which a number $n$ can be written in base $N$ using the digits from $L$.

\end{example}

\begin{example}\label{ex5.2}
Let $R=4$, $B=\{0,2\}$. Then $\mu$ is the Cantor 4-measure in Jorgensen-Pedersen \cite{JP98} or Picioroaga-Weber \cite{PW15}. 
Take $L=\{0,3,9\}$, and $\alpha_0=1$, $|\alpha_3|^2+|\alpha_9|^2=1$, $\alpha_3,\alpha_9\neq 0$. 

Note first that the sets 
$$E(0,3)=\{e_{l_0+4l_1+\dots +4^k l_k} : l_0\dots l_k\in\Omega(\{0,3\})\},\quad E(0,9)=\{e_{l_0+4l_1+\dots +4^kl_k}  : l_0\dots l_k\in\Omega(\{0,9\})\} $$
are incomplete. Indeed $(-1-3)/4=-1$ so $\{-1\}$ is an extreme cycle for the digits $\{0,3\}$, and $(-3-9)/3=-3$ so $\{-3\}$ is an extreme cycle for the digits $\{0,9\}$. This implies that the sets are incomplete (see \cite{DJ06}, or use Theorem \ref{th2} to see that $e_{-1}$ is orthogonal to $E(0,3)$ and $e_{-3}$ is orthogonal to $E(0,9)$. 

It is easy to check that condition (ii) in Theorem \ref{th1} is satisfied. We look now for the finite minimal invariant sets $M$ (for $L$). By Proposition \ref{pr3}, $M$ is contained in $[-9/3,0]$ and also $M$ is contained in $\frac12\bz$. But if $x$ in $M$ is of the form $x=\frac{2k+1}2$, $k\in\bz$, then the transition $x\rightarrow x/4$ is possible so $x/4$ should be in $M$ but it is not in $\frac12\bz$. Thus $M\subset \{-3,-2,-1,0\}$. Since the transition $-3\stackrel{3}{\rightarrow} -3/2$ is possible, we get that $-3$ cannot be in $M$. Similarly we have 
$-2\stackrel{0}{\rightarrow}-1/2$, $-1\stackrel{9}{\rightarrow}-5/2$ so $M$ cannot contain $-2,-1$, thus $M=\{0\}$.

With Theorem \ref{th2}, we obtain that 
$$\left\{\left(\prod_{j=0}^k \alpha_{l_j}\right)e_{l_0+4l_1+\dots +4^kl_k} : l_0\dots l_k\in\Omega(\{0,3,9\})\right\}$$
is a Parseval frame for $L^2(\mu)$.
\end{example}

More generally, we have the following theorem:

\begin{theorem}\label{th5.3}
In dimension $d=1$, let $|R|\geq 2$ and $B$ be as in Theorem \ref{th1} and assume also that the numbers $\alpha_{l}\neq 0$, $l\in L$ satisfy the conditions (i) and (ii) in Theorem \ref{th1}.
Assume that for all $l\in L$, $l\neq0$ we have 
\begin{equation}
[l]\cap L\neq \{l'\in L : l'\equiv l(\mod R)\}.
\label{eq5.3.1t}
\end{equation}
Then 
$$\left\{\left( \prod_{i=0}^k\alpha_{l_i}\right) e_{l_0+R l_1+\dots +R^{k}l_k} : l_0\dots l_k\in\Omega(L)\right\}$$
is a Parseval frame for $L^2(\mu(R,B))$.
\end{theorem}

\begin{proof}
The result follows from Theorem \ref{th2} and Proposition \ref{pr3}: if we have a minimal finite invariant set $M$, and we take a point $t\in M$ such that the transition $t\rightarrow g_{l_0}(t)$ is possible in one step, then we must have $[l_0]\cap L=\{l'\in L : l'\equiv l_0(\mod R)\}$. But, according to the hypothesis, this is possible only for $l_0=0$. Which means $t$ has to be the fixed point of $g_0$, so $M=\{0\}$. 
\end{proof}

\begin{corollary}\label{cor5.3}
Consider the Jorgensen-Pedersen measure $\mu_4$ associated to $R=4$, $B=2$. Let $L$ be a finite set that contains $0$ and the rest of the elements are odd numbers. Assume that one of the following conditions is satisfied:
\begin{enumerate}
	\item There exist numbers $l_0,l_0'$ in $L$ such that $l_0\equiv 1(\mod 4)$ and $l_0'\equiv 3(\mod 4)$. 
	\item There exist a number $l_0\in L$, $l_0\neq 0$ such that there are no extreme cycles for $(R,B,\{0,l_0\})$.
\end{enumerate}
Let $\{\alpha_l\}_{l\in L}$ be some non-zero complex numbers with $\alpha_0=1$ and $\sum_{l\in L}|\alpha_l|^2=2$.
 Then the set 

\begin{equation}
\left\{\left(\prod_{j=0}^k \alpha_{l_j}\right) e_{l_0+4l_1+\dots+4^kl_k} :l_0\dots l_k\in\Omega(L)\right\}
\label{eq5.3.2}
\end{equation}
is a Parseval frame for $L^2(\mu_4)$.
\end{corollary}

\begin{proof}
 It is easy to check that the conditions (i) and (ii) in Theorem \ref{th1} are satisfied. 

Under the assumption (i), the result follows from Theorem \ref{th5.3}: all the numbers in $L$ except 0 are odd, so $[l_0]\cap L=L\setminus \{0\}$ and $l_0\not\equiv l_0'\mod 4$. 

Under the assumption (ii), we use Corollary \ref{corbigger}, and we have that the set $\{0,l_0\}$ generates a {\it complete} orthonormal set, therefore the same is true for $L$. 
\end{proof}

\begin{remark}\label{rem5.4}
In \cite{DH16}, many examples of numbers $l_0$ are found with the property stated in (ii), that is, the set $\{0,l_0\}$ generates a complete orthonormal basis for $L^2(\mu_4)$. For example, if $p$ is a prime number $p\geq 5$, then the set $\{0,p^k\}$ generates an orthonormal basis, for any $k\geq1$. So if the set $L$ in Theorem \ref{th5.3} contains $p^k$ for some prime number $p\geq 5$ and $k\geq1$, then the set in \eqref{eq5.3.2} is a Parseval frame for $L^2(\mu_4)$.
\end{remark}

\begin{example}\label{ex5.5}
We consider next the case when the set $B$ has two digits. 

\begin{theorem}\label{th5.6}
Let $R$ be an integer $|R|\geq 2$, $B=\{0,b\}$ with $b\in \bz$. Let $L$ be a subset of $\bz$ with $0\in L$ and let $\{\alpha_l\}_{l\in L}$ be some non-zero complex numbers with $\alpha_0=1$. 
The condition (ii) in Theorem \ref{th1} is satisfied if and only if the following four conditions hold:
\begin{enumerate}
	\item $R=2^\alpha r$ with $\alpha\in\bn$ and $r$ odd;
	\item $b=2^\beta q$ with $\beta\in \bn\cup\{0\}$, $\beta\leq \alpha-1$ and $q$ odd;
	\item All numbers $l\in L\setminus\{0\}$ are of the form $l=2^{\alpha-1-\beta}s$ with $s$ odd and $qs$ divisible by $r$;
	\item $\sum_{l\in L\setminus \{0\}}|\alpha_l|^2=1$;	
\end{enumerate}
If in addition
\begin{enumerate}
	\item[(v)]  There exist numbers $l_1,l_2\in L\setminus\{0\}$ such that $l_1\not\equiv l_2(\mod R)$. 
\end{enumerate}
then 
\begin{equation}
\left\{\left(\prod_{j=0}^k \alpha_{l_j}\right) e_{l_0+Rl_1+\dots+R^kl_k} : l_0\dots l_k\in\Omega(L)\right\}
\label{eq5.6.1}
\end{equation}
is a Parseval frame for $L^2(\mu(R,B))$.

\end{theorem}

\begin{proof}
If the conditions (i)--(iv) hold then for $l\in L\setminus\{0\}$, with $l=2^{\alpha-1-\beta}s$, we have $\frac{bl}{R}=\frac{qs}{2r}$ and $\frac{qs}{r}$ is an odd number. Therefore $e^{2\pi i\frac{bl}{R}}=-1$, and it is easy to see that the conditions (i) and (ii) in Theorem \ref{th1} are satisfied. 

We also have, that for each $l_1,l_2\in L\setminus\{0\}$, $(l_1-l_2)b/R$ is the difference of two odd numbers over 2, so it is an integer; therefore $[l]\cap L=L\setminus\{0\}$ for all $l\in L\setminus\{0\}$. So, if (v) is also satisfied, then the result follows from Theorem \ref{th5.3}.

For the converse, if the conditions (i) and (ii) in Theorem \ref{th1} are satisfied, then we write $R=2^\alpha r$, $b=2^\beta q$ with $r$ and $q$ odd. With Proposition \ref{pr1.10}, we must have 
$1+e^{2\pi i bl/R}=0$ for all $l\in L\setminus\{0\}$, so $bl/R$ is of the form $(2k+1)/2$ for some $k\in\bz$. Thus $R$ has to be even. So $\alpha\geq 1$. If $l$ is of the form $l=2^\gamma s$, then 
we get that $\beta+\gamma-\alpha=-1$. So $\beta\leq \alpha-1$ and $\gamma=\alpha-1-\beta$. Also $qs/r=(2k+1)$ so (i)--(iii) hold. (iv) is also immediate. 
\end{proof}
\end{example}

\begin{remark}\label{rem5.7}
It is interesting to note that the condition (v) is enough to guarantee completeness of the resulting family of weighted exponential functions. It is known that in many cases, for orthonormal bases, if the digits in $L$ are not picked carefully then the resulting set of exponential functions is incomplete, see Example \ref{ex5.2}, or \cite{DH16} for some more such examples.  
\end{remark}

\begin{example}\label{ex5.8}
For the Jorgensen-Pedersen measure $\mu_4$, with $R=4$ and $B=\{0,2\}$. Let $L=\{0,3,15\}$. For this choice of the set $L$ we have a non-trivial, minimal invariant set $M=\{-1,-4\}$. Indeed, the only possible transitions in one step are $-1\stackrel{3}{\rightarrow}-1$, $-1\stackrel{15}{\rightarrow}-4$, and $-4\stackrel{0}{\rightarrow}-1$. By Theorem \ref{th2}, the corresponding set of weighted exponential functions is incomplete. 
\end{example}

\begin{example}\label{ex5.9}
We consider now the case when the set $B$ has three digits. 

\begin{theorem}\label{th5.10}
Let $R$ be an integer, $|R|\geq 2$, $B=\{0,b_1,b_2\}$ with $b_1,b_2\in\bz$. Let $L$ be a subset $\bz$ with $0\in L$ and let $\{\alpha_l\}_{l\in L}$ be some non-zero complex numbers with $\alpha_0=1$. The condition (ii) in Theorem \ref{th1} is satisfied if and only if the following four conditions hold:
\begin{enumerate}
	\item $R=3^\alpha r$, $\alpha\in\bn$, $r$ not divisible by 3;
	\item $b_i=3^\beta q_i$, $\beta\in\bn\cup\{0\}$, $\beta\leq \alpha-1$ and $q_i$ not divisible by 3, $i=1,2$, and $q_1\not\equiv q_2(\mod 3)$;
	\item All numbers $l\in L\setminus\{0\}$ are of the form $l=3^{\alpha-1-\beta}s_l$, with $s_l$ not divisible by 3 and $q_is_l$ divisible by $r$, $i=1,2$;
	\item The numbers $\alpha_l$ satisfy the equalities:
	$$\sum_{l\in L, s_l\equiv 1(\mod 3)}|\alpha_l|^2=1,\quad \sum_{l\in L, s_l\equiv 2(\mod 3)}|\alpha_l|^2=1;$$	
\end{enumerate}
If in addition
\begin{enumerate}
	\item[(v)] There exist $l_1,l_1',l_2,l_2'\in L\setminus\{0\}$ such that $s_{l_1}\equiv s_{l_1'}\equiv 1(\mod 3)$, $s_{l_2}\equiv s_{l_2'}\equiv 2(\mod 3)$ and $l_1\not\equiv l_1'(\mod R)$, $l_2\not\equiv l_2'(\mod R)$.  
\end{enumerate}
then the set 
\begin{equation}
\left\{\left(\prod_{j=0}^k \alpha_{l_j}\right) e_{l_0+Rl_1+\dots+R^kl_k} : l_j\in L, k\in\bn\right\}
\label{eq5.9.1}
\end{equation}
is a Parseval frame for $L^2(\mu(R,B))$.

Conversely, if the conditions (i) and (ii) in Theorem \ref{th1} are satisfied, then the statements in (i)--(iv) hold. 
\end{theorem}

\begin{proof}
Assume (i)--(iv) hold. We can assume $q_1\equiv 1$, $q_2\equiv 2$, $r\equiv 1(\mod 3)$, the other cases can be treated similarly. 

We have that $q_1s_l-s_lr$ is divisible by $r$ and by 3. Therefore $q_1s_l-s_lr=3rk$ for some $k\in\bz$, so $\frac{q_1s_l}{r}\equiv s_l(\mod 3)$. Similarly $\frac{q_2s_l}{r}\equiv 2s_l(\mod 3)$. 

Then
$$\sum_{l\in L}|\alpha_l|^2e^{2\pi i\frac{b_1l}{R}}=1+\sum_{l : s_l\equiv 1(\mod 3)}|\alpha_l|^2e^{2\pi i\frac{q_1s_l}{3r}} +\sum_{l : s_l\equiv 2(\mod 3)}|\alpha_l|^2e^{2\pi i\frac{q_1s_l}{3r}}$$
$$=1+e^{2\pi i\frac13}+e^{2\pi i\frac23}=0.$$

Similarly 
$$\sum_{l\in L}|\alpha_l|^2e^{2\pi i\frac{b_2l}{R}}=0,\quad \sum_{l\in L}|\alpha_l|^2e^{2\pi i\frac{(b_1-b_2)l}{R}}=0,\quad \sum_{l\in L}|\alpha_l|^2=3.$$
Thus, the conditions (i) and (ii) in Theorem \ref{th1} are satisfied. 

If $l_0\in L\setminus\{0\}$ then 
$$[l]\cap L=\{l\in L :\frac{q_i(s_l-s_{l_0})}{3r}\in\bz\}=\{l\in L :s_l\equiv s_{l_0}(\mod 3)\}.$$
Condition (v) implies that \eqref{eq5.3.1t} is satisfied, and the result then follows from Theorem \ref{th5.3}.

For the converse, write $b_1=3^{\beta_1}q_1$, $b_2=3^{\beta_2}q_2$, $R=3^\alpha r$, $l=3^{\gamma_l}s_l$ with $\alpha,\beta_1,\beta_2,\gamma_l\geq 0$ and $r,q_1,q_2,\gamma_l$ not divisible by 3 ($l\in L\setminus\{0\}$). 

Proposition \ref{pr1.10} implies that 
$$1+e^{2\pi i 3^{\beta_1+\gamma_l-\alpha}\frac{q_1 s_l}{r}}+ e^{2\pi i 3^{\beta_2+\gamma_l-\alpha}\frac{q_2 s_l}{r}}=0,$$
for all $l\in L\setminus\{0\}$. 

But it is easy to see that if $|z_1|=|z_2|=1$ and $1+z_1+z_2=0$ then $\{z_1,z_2\}=\{e^{2\pi i\frac13},e^{-2\pi i\frac13}\}$. 
Thus $3^{\beta_i+\gamma_l-\alpha}\frac{q_is_l}{r}=\pm\frac13+k$ for some $k\in\bz$. This implies, since $q_i,s_l,r$ are not divisible by 3, that $\beta_i+\gamma_l-\alpha=-1$ for all $l\in L\setminus\{0\}$ and $i=1,2$. So $\beta_1=\beta_2=:\beta\leq \alpha-1$ and $\gamma_l=\alpha-\beta-1$ for all $l\neq 0$. Furthermore, $\frac{q_is_l}{r}=\pm 1+3k$. So $\frac{q_is_l}{r}$ is an integer and $q_1\not\equiv q_2(\mod 3)$. 

We can assume $q_1\equiv 1,q_2\equiv 2(\mod 3)$ and $r\equiv 1(\mod 3)$, the other cases can be treated similarly. Since we have 
$$0=\sum_{l\in L}|\alpha_l|^2e^{2\pi i \frac{b_1l}{R}}=1+\left(\sum_{l\in L, s_l\equiv 1(\mod 3)}|\alpha_l|^2\right)e^{2\pi i \frac13}+\left(\sum_{l\in L, s_l\equiv 2(\mod 3)}|\alpha_l|^2\right)e^{2\pi i \frac23},$$
it follows that (iv) is satisfied. 
\end{proof}

\end{example}

\begin{example}\label{ex5.11}
Let $R=6$, $B=\{0,2,4\}$, $L=\{0,1,5,20\}$ and $\alpha_0=1$, $\alpha_1=1$ and $\alpha_5,\alpha_{20}$ some non-zero complex numbers with $|\alpha_5|^2+|\alpha_{20}|^2=1$. The conditions (i)--(iv) of Theorem \ref{th5.10} are satisfied, but condition (v) is not. However, the set generates a complete Parseval frame of weighted exponential functions. We prove that there are no finite minimal invariant sets other than $\{0\}$. 

Note that 
$$m_B(x)=\frac{1}{3}\left(1+e^{2\pi i 2 x}+e^{2\pi i4 x}\right),$$
therefore $m_B(x)=1$ on $\frac12\bz$ and $m_B(x)=0$ iff $x$ is of the form $x=\frac{6k+j}{6}$ with $k\in\bz$ and $j\in\{1,2,4,5\}$ (consider the roots of $1+z^2+z^4=(z^6-1)/(z^2-1)$). 

By Proposition \ref{pr3}, if $M$ is a minimal finite invariant set, then every $x\in M$ is in $\frac12\bz$. But if $x=\frac{2k+1}{2}$ for some integer $k$, then the transition  $x\stackrel{0}{\rightarrow}\frac x6=\frac{2k+1}{12}\not\in\frac12\bz$ is possible in one step, so, such an element $x$ cannot be in $M$. 

So, again by Proposition \ref{pr3}, the only possible candidates for elements in $M$ are $-4,-3,-2,-1$. The following transitions are possible in one step $-4\stackrel{20}{\rightarrow}-4$, but also 
$-4\stackrel{5}{\rightarrow}-\frac96=-\frac32$, so $-4$ is not in $M$. $-3\stackrel{0}{\rightarrow}-\frac12$ so $-3$ is not in $M$. $-2\stackrel{1}{\rightarrow}-\frac12$ so $-2$ is not in $M$. 
$-1\stackrel{5}{\rightarrow}-1$ but also $-1\stackrel{20}{\rightarrow}-\frac{-21}{6}=-\frac72$ so $-1$ is not in $M$. Thus $M$ has to be $\{0\}$. 

With Theorem \ref{th2} we obtain that the set 
$$
\left\{\left( \prod_{i=0}^k\alpha_{l_i}\right) e_{l_0+R l_1+\dots +R^{k}l_k} : l_0\dots l_k\in\Omega(L)\right\}$$
is a Parseval frame for $L^2(\mu(R,B))$.

We note also that every proper subset of $L$ does not generate Parseval frames (even with other choices of the numbers $\alpha_l$). $\{0,1,5\}$ has an extreme cycle $\{-1\}$, $\{0,1,20\}$ has an extreme cycle $\{-4\}$, $\{0,5,20\}$ does not satisfy the condition (ii) in Theorem \ref{th1}, but the argument is that $e_1$ is orthogonal to $e_{l_0+6k}$ for $l_0\in\{0,5,20\}$ and $k$ in $\bz$.
Indeed, note that for all such $l_0$ and $k$,
$$m_B\left(\frac{1-(l_0+6k)}{6}\right)=0,$$
Therefore 
$$\ip{e_1}{e_{l_0+6k}}=\widehat\mu(1-(l_0+6k))=\prod_{n=1}^\infty m_B\left(\frac{1-(l_0+6k)}{6^n}\right)=0.$$

Thus the family of exponential functions generated by $\{0,5,20\}$ cannot be complete. 
\end{example}

We end this section with a conjecture. 

\begin{conjecture}\label{con1}
Let $R$ be a $d\times d$ expansive integer matrix, and let $B$ be a set in $\bz^d$ with $0$ in $B$ and such that the elements of $B$ are mutually incongruent modulo $R$. Suppose there exist a set $L$ in $\bz^d$ with $0\in L$, and numbers $\{\alpha_l : l\in L\}$ with $\alpha_0=1$ such that the functions $\{\alpha_l e_l : l\in L\}$ form a Parseval frame for $L^2(\delta_B)$. Then there exists a set $L_0$ in $\bz^d$, with $0\in L_0$ such that the functions $\{e_l : l\in L_0\}$ form an orthonormal basis for $L^2(\delta_B)$ (in other words $R^{-1}B$ is a spectral set, with spectrum in $\bz^d$).
\end{conjecture}

Propositions \ref{pr1.10} and \ref{preq} offer some positive evidence to support this conjecture.

\begin{acknowledgements}
This work was partially supported by a grant from the Simons Foundation (\#228539 to Dorin Dutkay).

\end{acknowledgements}
\bibliographystyle{alpha}
\bibliography{eframes}

\end{document}